\documentclass[a4paper,11pt]{article}

\usepackage{amsfonts}
\usepackage{amssymb}
\usepackage{amsmath}
\usepackage{amsthm}
\usepackage{epsfig}
\usepackage{mathrsfs}
\usepackage{graphicx}
\usepackage{dsfont}
\usepackage{epsfig,color}
\usepackage{dsfont}
\usepackage[utf8]{inputenc}
\newcommand{\R}{\mathbb{R}}
\newcommand{\Rd}{{\mathbb{R}^d}}

\newcommand{\E}{\mathcal{E}}
\newcommand{\I}{\mathcal{I}}
\newcommand{\norm}[1]{\left\| #1 \right\|}
\newcommand{\meni}{\leqslant}
\newcommand{\maig}{\geqslant}
\newcommand{\supp}{\textrm{supp}}

\newtheorem{definicion}{Definition}[section]
\newtheorem{proposition}[definicion]{Proposition}
\newtheorem{theorem}[definicion]{Theorem}
\newtheorem{corollary}[definicion]{Corollary}
\newtheorem{lemma}[definicion]{Lemma}
\newtheorem{remark}[definicion]{Remark}

\numberwithin{equation}{section}

\begin{document}

\title{Exponential Convergence Towards Stationary States \\for the 1D  Porous Medium Equation \\with Fractional Pressure}

\author{\textsc{J. A. Carrillo\thanks{Department of Mathematics, Imperial College
London, London SW7 2AZ, UK. Email: carrillo@imperial.ac.uk.}, Y. Huang\thanks{Department of Mathematics, Imperial College London, London SW7 2AZ, UK. Email: yanghong.huang@imperial.ac.uk.}, M. C. Santos\thanks{Departamento de Matem\'atica - IMECC, Universidade Estadual de Campinas, 13083-859, Campinas-SP, Brazil. Email: ra115906@ime.unicamp.br.}, J. L. V\'{a}zquez\thanks{Departmento de Matem\'aticas, Universidad Aut\'onoma de Madrid, 28049-Madrid, Spain. Email: juanluis.vazquez@uam.es.}}}
\date{}
\maketitle

\begin{abstract}
We analyse the asymptotic behaviour of solutions to the one dimensional fractional
version of the porous medium equation introduced by Caffarelli and V\'{a}zquez~\cite{vazquez2,vazquez1}, where the pressure is obtained as a Riesz potential associated to the density. We take advantage of
the displacement convexity of the Riesz potential in one dimension to show a functional inequality involving the entropy, entropy dissipation, and the Euclidean transport distance.
An argument by approximation shows that this functional inequality is enough to deduce the exponential convergence of solutions in self-similar variables to the unique steady states.
\end{abstract}

\footnotetext{\emph{Date: \today}}

%26A33, Fractional derivatives and integrals
%35A05, General existence and uniqueness theorems
%35K55, Nonlinear PDE of parabolic type
%35K65, %Parabolic partial differential equations of degenerate type
%35S10, %Initial value problems for PsDO
%76S05 %Flows in porous media; filtration; seepage


%============================================================================================================
%=======================================================INTRODUCTION=========================================
\section{Introduction}

In this work, we analyse the long-time asymptotics of the nonlinear nonlocal equation
\begin{equation}\label{eq:fracPMESM}
\rho_t = \nabla\cdot\big(
\rho(\nabla (-\Delta)^{-s} \rho + \lambda x)
\big), \qquad \lambda>0,\ x\in\mathbb{R}^d\,,
\end{equation}
obtained from the fractional version of the porous medium equation introduced by Caffarelli and V\'{a}zquez~\cite{vazquez2,vazquez1}
\begin{equation}\label{eq:fracPME}
u_\tau = \nabla\cdot (u\nabla p),\quad p = (-\Delta)^{-s} u \,,
\end{equation}
by passing to self-similar variables. Indeed, by adding the Fokker-Planck confining term $\nabla \cdot (xu)$, solutions to~\eqref{eq:fracPMESM} will characterize the long-time asymptotic behaviour
of solutions to~\eqref{eq:fracPME}. This connection will be further explained below.

The fractional porous medium equation~\eqref{eq:fracPME}  can be viewed as a continuity equation, $u_\tau + \nabla\cdot(u \mathbf{V})=0$,  for a density or concentration $u(\tau,y)$ with velocity $\mathbf{V} = -\nabla p$, where the  velocity potential or pressure $p$  is related to $u$ by the inverse of a fractional Laplacian operator  $p=(-\Delta)^{-s} u$, $0<s<1$. The standard porous medium equation is recovered for $s=0$. We assume that the unknown $u(\tau,y)$, representing a density or concentration, is defined for $y \in \mathbb{R}^d$
and $\tau>0$ and supply initial data $u(y,0)=u_0(y)$,  a nonnegative mass distribution in $L^1(\mathbb{R}^d)\cap L^\infty(\mathbb{R}^d)$.
We also point out that the pressure  can be represented as
\begin{equation*}%\label{eq:slap}
 p= (-\Delta)^{-s} u = W\ast u,
\end{equation*}
with the singular convolution kernel
\begin{equation}\label{eq:kernelW}
 \qquad W(y)=c_{d,s}|y|^{2s-d},\qquad
c_{d,s} = \frac{s2^{-2s}\Gamma(d/2-s)}{\pi^{d/2}\Gamma(1+s)},
\end{equation}
and $0<s<\min(1,d/2)$,
called the Riesz potential of $u$ as in the standard textbooks~\cite{MR0350027,MR0290095}. This representation also makes sense for $s=d/2$ with the logarithm kernel $W(y)=-2^{1-d}\pi^{-d/2}\Gamma(d/2)^{-1}\log |y|$ (see \cite{carr1,ledoux} in one dimension) and for $1/2 < s<1$ in one dimension with the negative coefficient
$c_{1,s}$ and the positive exponent $2s-1$ in $W(y)$. As a result, the potential
$W$ does not necessarily decay to zero at infinity in the last two cases, but the magnitude of the gradient $\nabla  W$ does.
When the kernel $W(y)$ is replaced by a less singular radially symmetric function, the same equation appeared in granular flow~\cite{MR1471181,MR1812737,LT,MR2053570} and biological swarming~\cite{MR1698215,BCL,BT2}.

To describe the long time behaviour of solutions to~\eqref{eq:fracPME}, it is more convenient to study the corresponding transformed equation~\eqref{eq:fracPMESM} as discussed in~\cite{MR1777035,vazquez1}, by defining
\begin{equation}\label{changevariables}
\rho(t,x) := (1+\tau)^{\alpha} u(\tau,y),
\end{equation}
with the similarity variables $x=y(1+\tau)^{-\beta}$ and $t = \log(1+\tau)$.
The exponents $\alpha$ and $\beta$ can be determined from dimensional analysis
and the mass conservation~\cite{MR1426127}, which are given by
\begin{equation}\label{changevariables2}
\alpha = d/(d+2-2s),\quad \beta = 1/(d+2-2s).
\end{equation}
In this way, the rescaled density $\rho(t,x)$ satisfies~\eqref{eq:fracPMESM} with $\lambda=\beta=1/(d+2-2s)$.  We will keep $\lambda>0$ arbitrary in \eqref{eq:fracPMESM} as a parameter to characterize the convexity of the energy defined below and the convergence rate to the steady state later on. As a result, the long time behaviour of the original density $u(\tau,y)$ is completely specified if we establish the convergence of $\rho(t,x)$ to the steady state $\rho_\infty(x)$ of~\eqref{eq:fracPMESM} with
$\lambda=\beta$.

The existence and uniqueness of the steady state $\rho_\infty$ of~\eqref{eq:fracPMESM} for each given mass was initially characterized by an obstacle problem in~\cite{vazquez1}, and then the explicit
expression of $\rho_\infty$ was obtained by Biler, Imbert and Karch~\cite{biler2011barenblatt,biler}, for even more general nonlinear dependence of the pressure $p=(-\Delta)^{-s}u^{m-1}$, $m>1$.
In case $m=2$ of our interest here, the self-similar solution of~\eqref{eq:fracPME} is given by
\[
u(\tau,y) = (1+\tau)^{-d/(d+2-2s)} \rho_\infty\big(y(1+\tau)^{-1/(d+2-2s)}\big),
\]
with the self-similar profile
\[
\rho_\infty(x)
= K_{d,s}\big(R^2-|x|^2\big)_+^{1-s}
\]
and the prefactor
\begin{equation*}%\label{eq:constd}
    K_{d,s} =\frac{2^{2s-1}\Gamma(d/2+1)}{\Gamma(2-s)\Gamma(d/2+1-s)}\lambda\,.
\end{equation*}
The radius of the support $R$ is determined by the total conserved mass $M$, that is,
\begin{equation}\label{eq:MR}
M = \int_{\mathbb{R}^d} u(\tau,y)dy = \frac{2^{2s}\pi^{d/2}\Gamma(d/2+1)\lambda}{(d+2-2s)\Gamma(d/2+1-s)^2}R^{d+2-2s}.
\end{equation}
After these preliminary discussion, we concentrate on the convergence of $\rho(t,x)$ to the steady state $\rho_\infty(x)$ in the rest of the paper.

Let us point out that the fractional porous medium equation \eqref{eq:fracPMESM} can be viewed as a particular case of the aggregation equation \cite{MR2053570,BCL,BCLR2} written as
\begin{equation}\label{eq:aggreg}
\rho_t = \nabla\cdot\big( \rho ( \nabla W\ast \rho+\nabla V ) \big), \qquad x\in\mathbb{R}^d\,,
\end{equation}
where $V(x)=\tfrac{\lambda}2 |x|^2$ and $W (x)=c_{d,s} |x|^{2s-d}$, $0<s<1$.

During the past fifteen years, several important techniques~\cite{otto,MR1777035,dolbeault,MR2053570,villani,ambrosiogiglisavare}
have been developed for the convergence of linear or nonlinear Fokker-Planck equations to their steady states with sharp rate.
These techniques can also  be employed to prove the convergence of solutions of~\eqref{eq:fracPMESM} to $\rho_\infty$, by realizing that the free energy $\E(\rho)$ defined as
\begin{eqnarray}\label{Energy}
\E(\rho)&=&  \frac{1}{2}\int_\Rd \big\{ (-\Delta)^{-s}\rho(x) + \lambda|x|^2 \big\}\rho(x)\;dx \\
&=&\frac{c_{d,s}}{2}\int_\Rd \int_\Rd \frac{\rho(x)\rho(y)}{|x-y|^{d-2s}}\;dydx+ \lambda\int_\Rd \frac{|x|^2}{2}\rho(x)\;dx, \nonumber
\end{eqnarray}
is a Lyapunov functional for $0<s<\min(1,d/2)$. One can similarly define the Lyapunov functional for $1/2\leq s< 1$ in one dimension, assuming that $\rho$ satisfies a growth condition at infinity, namely $\rho\log|x|  \in L^1(\R)$ if $s=1/2$ and $\rho|x|^{2s-1} \in L^1(\R)$ if $1/2 <s<1$. In fact, \eqref{eq:fracPMESM} is a gradient flow of the free energy functional \eqref{Energy} with respect to the Euclidean transport distance in the metric space of probability measures~\cite{ambrosiogiglisavare,MR2209130}.

The basic properties of the energy $\E(\rho)$ and its dissipation $\I(\rho)$ defined below, together with the long-time asymptotics of
solutions to \eqref{eq:fracPMESM}, are already derived in~\cite{vazquez1}.
More precisely, along the evolution governed by~\eqref{eq:fracPMESM}, one can obtain the formal relation
$d\E(\rho)/dt=-\I(\rho)$, where we denote by $\I(\rho)$ the entropy production or entropy dissipation of $\E$ given by
\[
\qquad \I(\rho) = \int_\Rd \rho \left| \nabla \xi\right|^2dx\;,\;\; \mbox{ with }\;\;\xi = \frac{\delta \E}{\delta \rho} = (-\Delta)^{-s}\rho + \frac{\lambda}{2}|x|^2.
\]
Using this relation, the solution of~\eqref{eq:fracPMESM}  is shown
to converge towards $\rho_\infty$ in \cite{vazquez1}, but no rate is obtained. To be more precise, they show that solutions of the fractional porous medium equation~\eqref{eq:fracPMESM} satisfy the energy inequality
$\E\big(\rho(t,\cdot)\big)+\int_0^t \I\big(\rho(\tau,\cdot)\big)d\tau\leq \E\big(\rho(0,\cdot)\big)$
 that is enough to conclude the converge of $\rho(t,x)$ to the steady state $\rho_\infty(x)$.

In this work, we will focus on obtaining the sharp convergence rate  for the solutions of the Cauchy problem for \eqref{eq:fracPMESM} towards the equilibrium $\rho_\infty$, for all $0<s<1$  in one dimension,
although many of the calculations are presented in general dimensions.
In the particular case of $s=1/2$ in one dimension, the kernel is given by the logarithmic potential and it was treated in \cite{carr1}, see also \cite{ledoux} for related functional inequalities. In fact, it is shown in \cite{carr1} that the energy $\E(\rho)$ is displacement convex, which can not be derived directly from the criteria given in the seminar paper by McCann~\cite{MR1451422}. We will take advantage of these techniques in~\cite{carr1} to prove certain functional inequalities, in particular the HWI inequalities as introduced in~\cite{ottovillani} (also obtained in~\cite{ledoux} for the logarithmic case $s=1/2$).
This displacement convexity and related inequalities are then used to show the convergence towards equilibrium in one dimension, through the exponential decay of the transport distances and the relative energy,
for general $s \in (0,1)$.

Finally, we point out that the problem of sharp convergence rates in several space dimensions is still open. Moreover, it could be interesting to prove or disprove analogous functional inequalities involving nonlocal operators in several space dimensions corresponding to the ones established here in one dimension; see more comments at the end of Section 2. New techniques or inequalities have to be developed. Showing asymptotic convergence when the confining term $\nabla\cdot(\lambda x\rho)$ is replace by the general drift $\nabla\cdot(\rho\nabla V)$ is another interesting problem, see \cite{chafai,MR2053570}.

The organization of this work is as follows. We first remind the reader in Section 2 about the basics of the entropy/entropy dissipation method, together with the main functional inequality that we will prove in one dimension. In fact, we follow closely the strategy developed for nonlinear diffusion equations in~\cite{MR889476,MR1842428,MR1777035,dolbeault,MR1853037,MR2053570} to reduce to the proof of a Log-Sobolev type inequality. This inequality is then proved in Section 3 as a consequence of the HWI inequality which crucially uses the displacement convexity. Finally, Section 4 is devoted to obtain the rate of convergence towards equilibrium of the solutions to \eqref{eq:fracPMESM} by an approximation method using the construction of solutions in \cite{vazquez2}.

%===============================================================================================================================

%===============================================================================================================================
\section{Entropy dissipation method}\label{BE}

In this section, we first show some formal computations using the Bakry-Emery strategy \cite{MR889476} demonstrating that the relative entropy $\E(\rho|\rho_\infty):=\E(\rho)-\E(\rho_\infty)$ decays to zero exponentially fast in one dimension, by taking the second order time derivative of $\E(\rho|\rho_\infty)$ along the evolution equation~\eqref{eq:fracPMESM}. We will then discuss the strategy we use to render this computation rigorous in the following sections.

Before starting the computations on the dissipation of the free energy, let us discuss a bit more on the equilibrium solution $\rho_\infty$.  It was recently proved in \cite[Theorem 1.2]{chafai} that $\E$ restricted to $\mathcal{P}(\R^d)$ is strictly convex in the classic sense for $0<s<\min(1,d/2)$, and it has a unique compactly supported minimizer $\rho_\infty$ characterized by
\begin{subequations}
\begin{align}
(-\Delta)^{-s}\rho_\infty(x) + \lambda\frac{|x^2|}{2} &= C_* \;, \;\;\; \forall\;x\in \supp(\rho_\infty)\label{min1}
\\
(-\Delta)^{-s}\rho_\infty(x) + \lambda\frac{|x^2|}{2} &\maig C_* \; , \;\;\; \mbox{a.e. }\Rd \,,\label{min2}
\end{align}
\end{subequations}
for some constant $C_*$ determined by the total mass. This formulation is equivalent to the obstacle problem in~\cite{vazquez1}, for the rescaled pressure $P = (-\Delta)^{-s}\rho$ and the quadratic obstacle $\Phi(x)=C_*-\tfrac{\lambda}2 |x|^2$.
Using the following relation (see~\cite{biler2011barenblatt,biler})
\begin{align}\label{rhoinfty}
  (-\Delta)^{-s}(R^2-|x|^2)_+^{1-s} &= \frac{2^{-2s}\Gamma(2-s)
 \Gamma(d/2-s)}{\Gamma(d/2)}\left(R^2-\frac{d-2s}{d}|x|^2\right)\cr
 &= \frac{\lambda}{2K_{d,s}}\left(\frac{d}{d-2s}R^2 - |x|^2\right) \;\;,\;\;\; \mbox{ for all } |x|\meni R,
\end{align}
it is easy to verify that $\rho_\infty=K_{d,s}(R^2-|x|^2)^{1-s}_+$
is indeed the minimizer for $\E$ for $0<s<\min(1,d/2)$. Similar computations can be done in the range $1/2\leq s <1$, see \cite{carr1,BCLR2} for instance.

Now, we can consider the difference $\E(\rho|\rho_\infty):=\E(\rho)-\E(\rho_\infty)$ as a measure of convergence towards equilibrium. We first rewrite the equation~\eqref{eq:fracPMESM} as
\begin{equation}
\rho_t = \nabla\cdot(\rho \nabla \xi)\qquad \mbox{ with } \ \xi := (-\Delta)^{-s}\rho + \lambda|x|^2/2.
\end{equation}
Assuming that $\rho$ (and thus $\xi$) is smooth enough, taking the time derivative of the entropy dissipation rate $\I(\rho)$ along the evolution equation, we obtain
\begin{align*}
\frac{d}{dt}\I(\rho) &= \int \rho_t |\nabla \xi|^2 + 2\int \rho \nabla\xi\cdot\nabla \xi_t\\
&= \int \nabla\cdot(\rho \nabla \xi)|\nabla \xi|^2
+2\int \rho \nabla\xi\cdot\nabla\big[(-\Delta)^{-s}
 \big(\nabla\cdot(\rho \nabla \xi)\big)\big].
\end{align*}
Using the fact $D^2\xi = D^2(-\Delta)^{-s}\rho + \lambda I$ for the Hessian matrix of $\xi$, the first term on the right hand side above can be written as
\[
 \int \nabla\cdot(\rho \nabla \xi)|\nabla \xi|^2
= -2\int \rho \langle D^2 \xi\cdot \nabla \xi, \nabla \xi \rangle
=-2\lambda \I(\rho) -2\int \rho \big\langle D^2(-\Delta)^{-s}\rho\cdot
\nabla \xi,\nabla \xi\big\rangle \,.
\]
Therefore, $d\I(\rho)/dt=-2\lambda \I(\rho) - 2\mathcal{R}(\rho)$ with
\begin{align}\label{eq:Rexp}
 \mathcal{R}(\rho) &= \int \rho \big\langle D^2(-\Delta)^{-s}\rho\cdot
\nabla \xi,\nabla \xi\big\rangle-
\int \rho \nabla\xi\cdot\nabla\big[(-\Delta)^{-s}
 \big(\nabla\cdot(\rho \nabla \xi)\big)\big].
\end{align}
The entropy-entropy dissipation method can be summarized as follows: if $\mathcal{R}(\rho)\geq0$
for the solution $\rho$, then from the conditions $d\E(\rho)/dt = -\I(\rho)$ and
$d\I(\rho)/dt \leq -2\lambda \I(\rho)$, we can conclude that
$\I(\rho)(t) \leq \I(\rho)(0)e^{-2\lambda t}$ and $\E(\rho)(t)-\E(\rho_\infty)
\leq \big( \E(\rho)(0)-\E(\rho_\infty)\big)e^{-2\lambda t}$, or the exponential convergence
of both $\I(\rho)(t)$ and $\E(\rho)(t)-\E(\rho_\infty)$ towards zero.

When $s=0$, the equation~\eqref{eq:fracPMESM} reduces to the standard porous medium equation with
quadratic nonlinearity. In this special case, the non-negativity of $\mathcal{R}(\rho)$
 was established in~\cite{MR1777035} using
several integration by parts, leading to (with $\xi=\rho+\lambda |x|^2/2$)
\[
\qquad  \mathcal{R}(\rho) =
\frac{1}{2}\int \rho^2\big[ (\Delta\xi)^2 + \|D^2\xi\|_F^2\big] \geq 0.
\]
Here $\|A\|_F = \sqrt{\mbox{tr} (A^TA)}$ is the Frobenius norm of the matrix $A$. Consequently, by deducing various decay on the norms of $\rho(t,\cdot)-\rho_\infty(\cdot)$,
the solution $\rho$ converges to its steady state exponentially fast.

However, in the case $s\in(0,1)$ considered here,
it is not immediately clear whether $\mathcal{R}(\rho)$ given
in~\eqref{eq:Rexp} above is
nonnegative or not. To simplify $\mathcal{R}(\rho)$, we need more explicit expressions
of $D^2(-\Delta)^{-s}\rho$ and $\nabla\big[(-\Delta)^{-s}
 \big(\nabla\cdot(\rho \nabla \xi)\big)\big]$, or the second order derivatives of the
Riesz potential of $\rho$ and $\rho\nabla\xi$ respectively. Since these derivatives
can not be applied to the corresponding kernel $W(x)=c_{d,s}|x|^{2s-d}$ directly, we
have to invoke the following technical lemma.

\begin{lemma} If $\rho$ is a smooth function on $\mathbb{R}^d$, then
the components of the Hessian matrix of the Riesz potential $(-\Delta)^{-s}\rho$ are
given by
\begin{equation}\label{eq:hessIF}
 D_{ij}(-\Delta)^{-s}\rho(x) =
%\frac{\partial^2}{\partial x_i\partial x_j}
\partial_{ij}(-\Delta)^{-s}\rho(x)
=-c^+_{d,s}\int K_{ij}(x-y)\big(\rho(x)-\rho(y)\big)dy,
\end{equation}
where $K_{ij}(x) = |x|^{2s-2-d}\big( (d+2-2s)x_ix_j/|x|^2-\delta_{ij}
\big)$ and $c^+_{d,s} = (d-2s)c_{d,s}$.
\end{lemma}
This lemma is proved by interpreting $D_{ij}(-\Delta)^{-s}\rho$ as
a distributional derivative, and the details are given in Appendix~\ref{sec:distdev}. Using the singular integral representation~\eqref{eq:hessIF}, we obtain
\begin{align}\label{rrho}
\mathcal{R}(\rho) &=
\sum_{i,j}\int_{\mathbb{R}^d}
\Big\{\rho(x)\partial_i\xi(x)\partial_j\xi(x) D_{ij}(-\Delta)^{-s}\rho(x)
-\rho(x)\partial_i\xi(x)D_{ij}(-\Delta)^{-s}[\rho\partial_j\xi](x)
\Big\}dx\cr
&=-c_{d,s}^+\sum_{i,j}\int_{\mathbb{R}^d}\int_{\mathbb{R}^d}
\rho(x)\partial_i\xi(x)K_{ij}(x-y) \cr
&\qquad \qquad \qquad \Big\{
\partial_j\xi(x)\big(\rho(x)-\rho(y))
-\rho(x)\partial_j\xi(x)+\rho(y)\partial_j\xi(y)
\Big\}dydx \cr
&= c_{d,s}^+\sum_{i,j}\int_{\mathbb{R}^d}\int_{\mathbb{R}^d}
\rho(x)\rho(y)\partial_i\xi(x)K_{ij}(x-y) \Big\{\partial_j\xi(x)-\partial_j\xi(y)\Big\}dydx \cr
&= \frac{c_{d,s}^+}{2}\int_{\mathbb{R}^d}\int_{\mathbb{R}^d}
\rho(x)\rho(y) \big\langle \nabla\xi(x)-\nabla \xi(y), \mathbf{K}(x-y)
\big(\nabla\xi(x)-\nabla \xi(y)\big)\big\rangle dydx,
\end{align}
where $\mathbf{K}(x)$ is a matrix with entries $K_{ij}(x)$ and the integrand is symmetrized in the last step.

\begin{remark} {\rm Similar expressions already appear in the context of non-local equations for granular flow or biological swarms, when the interaction kernel is smoother. In fact, if $\rho$ is a smooth solution of $\rho_t = \nabla\cdot(\rho \nabla W*\rho)$ with a smooth kernel $W$
such that the Hessian $D^2W$ is locally integrable, then the time derivative of the interaction energy $\frac{1}{2}\iint W(x-y)\rho(x)\rho(y)dydx$ is $-\int \rho|\nabla\xi|^2dx$ with $\xi = W*\rho$. The second order time derivative of the energy is
\[
-\int_\Rd \rho_t|\nabla \xi|^2 dx -2\int_\Rd \rho\nabla \xi\cdot \nabla \xi_t dx
\]
which is exactly
\[ \int_\Rd\int_\Rd \rho(x)\rho(y)\Big\langle D^2W(x-y)\big(\nabla\xi(x)-\nabla\xi(y)\big),\nabla\xi(x)-\nabla\xi(y)
\Big\rangle dy\,dx
\]
by applying appropriate integration by parts.}
\end{remark}

In one dimension, $\mathbf{K}(x)=(2-2s)|x|^{2s-3}$
is a positive scalar and $\mathcal{R}(\rho) \geq 0$ for
any non-negative density $\rho$, leading to the desired exponential convergence. However, in higher dimensions,
the matrix $\mathbf{K}(x)$ can be written as
\[
\mathbf{K}(x) = |x|^{2s-2-d}\big( (d+2-2s)x\otimes x/|x|^2-I)\,,
\]
which has one positive eigenvalue $\lambda_1=(d+1-2s)|x|^{2s-d-2}$
and $d-1$ negative eigenvalues $\lambda_i=-|x|^{2s-d-2}$, $i=2,\cdots,d$. Therefore,
it is not known from~\eqref{rrho} whether $\mathcal{R}(\rho)$ is positive or not.
To summarize, we can conclude that both the relative entropy $\E(\rho)-\E(\rho_\infty)$
and the entropy dissipation rate $\I(\rho)$ converge formally to zero
exponentially fast only in one dimension.

The above approach for the exponential decay in one dimension
can be proved rigorously, by establishing the results for mollified solutions to the regularized equation (with linear diffusion for example). One of the main difficulties in our case lies in the definition and continuity of the entropy dissipation $\I(\rho)$. The set of functions for which $\I$ is finite is difficult to handle. Therefore, passing to the limit the exponential decay of the entropy dissipation using density argument is a complicated task in our case.
Alternatively, we prove the same results in Section 3 for smooth
solutions, and then pass to the limit in Section 4. Before going
to that, we point out that the exponential convergence
is in fact intimately connected with certain inequalities in
the next subsection.

\subsection{Sobolev inequalities with fractional Laplacian}

If $\mathcal{R}(\rho) \geq 0$, from the limits $\I\big(\rho(t)\big)\to 0$
and $\E\big(\rho(t)\big)-\E(\rho_\infty\big) \to 0$ as $t$ goes to infinity and the inequality
\[
 \frac{d}{dt} \I(\rho) \leq -2\lambda \I(\rho)
 =-2\lambda \frac{d}{dt}\big(\E(\rho)-\E(\rho_\infty)\big),
\]
we can integrate in time to get
\begin{equation}\label{LSI}
\E(\rho)-\E(\rho_\infty) \meni \frac{1}{2\lambda}\I(\rho).
\end{equation}
On the other hand, by assuming~\eqref{LSI} above, we can also prove the
exponential convergence of $\E(\rho)-\E(\rho_\infty)$ to zero
with exponential rate $-2\lambda$ (but not necessarily
the exponential convergence of $\I(\rho)$), by integrating
\[
 \frac{d}{dt} \big(\E(\rho)-\E(\rho_\infty)\big)
=-\I(\rho) \meni -2\lambda \big(\E(\rho)-\E(\rho_\infty)\big)
\]
in time. The inequality~\eqref{LSI} is usually  called, in the context of optimal transport, Log-Sobolev inequality in the linear diffusion case or generalized Log-Sobolev inequalities otherwise. We will revisit~\eqref{LSI} in the next section by investigating the displacement convexity of the energy $\E(\rho)$.
In particular, it becomes the logarithmic Sobolev inequality~\cite{MR0420249} for linear
Fokker-Planck equation~\cite{MR1842428,MR1639292,MR1704435}, and a special family of {G}agliardo-{N}irenberg inequalities for nonlinear Fokker-Planck equations with porous medium type diffusion~\cite{dolbeault,MR1777035,MR1853037}.

Following a similar approach as Del Pino and Bolbeault~\cite{dolbeault},  expanding both sides of~\eqref{LSI}, we obtain the equivalent inequality
\begin{multline*}
 \lambda\left[
\int_{\mathbb{R}^d} \rho(x)(-\Delta)^{-s}\rho(x)dx
-2\int_{\mathbb{R}^d} \rho(x)
x\cdot\nabla(-\Delta)^{-s}\rho(x)dx
\right] \\
\leq 2\lambda \E(\rho_\infty)+
\int_{\mathbb{R}^d} \rho(x)|\nabla(-\Delta)^{-s}\rho(x)|^2dx.
\end{multline*}
The second term on the left-hand side can be simplified
using the definition of $(-\Delta)^{-s}\rho$ as
the Riesz integral
\[
 (-\Delta)^{-s}\rho(x) = c_{d,s}
\int_{\mathbb{R}^d} \frac{1}{|x-y|^{d-2s}}\;\rho(y)dy\,,
\]
and consequently
\begin{align*}
 -2\int_{\mathbb{R}^d} \rho(x)
x\cdot\nabla(-\Delta)^{-s}\rho(x)dx
 &=2(d-2s)c_{d,s}\int_{\mathbb{R}^d}\int_{\mathbb{R}^d}
\rho(x)\rho(y)x\cdot (x-y) |x-y|^{2s-d-2}dydx \cr
&= (d-2s)c_{d,s}\int_{\mathbb{R}^d}\int_{\mathbb{R}^d}
\rho(x)\rho(y) |x-y|^{2s-d}dydx  \cr
&= (d-2s)\int_{\mathbb{R}^d}\rho(x)(-\Delta)^{-s} \rho(x)dx.
\end{align*}
Therefore, the inequality~\eqref{LSI} becomes
\begin{equation}\label{eq:LSI1}
\lambda(d+1-2s)\int_{\mathbb{R}^d}\rho(x)(-\Delta)^{-s} \rho(x)dx
\leq 2\lambda \E(\rho_\infty)+
\int_{\mathbb{R}^d} \rho(x)|\nabla(-\Delta)^{-s}\rho(x)|^2dx.
\end{equation}
To get a self-consistent inequality, we have to write $\E(\rho_\infty)$
in terms of some functionals of $\rho$, which is established through the total conserved mass, $M=\int \rho = \int \rho_\infty$.
Using the explicit expression for $\rho_\infty(x)=K_{d,s}(R^2-|x|^2)_+^{1-s}$, the identity~\eqref{rhoinfty}
implies that
\[
 (-\Delta)^{-s}\rho_\infty(x) = \frac{\lambda}{2}\frac{d}{d-2s}R^2 - \frac{\lambda}{2}|x|^2 \;,\;\; \mbox{ for } |x|\meni R.
\]
Therefore, we conclude that
\begin{align*}
 \E(\rho_\infty) &= \frac{1}{2}\!\int_{\mathbb{R}^d}
\rho_\infty (x)\Big((-\Delta)^{-s}\rho_\infty(x) +\lambda |x|^2\Big) dx = \frac{\lambda K_{d,s}}{4}\int_{\mathbb{R}^d}\!
 (R^2-|x|^2)^{1-s}_+\!\left(
\frac{d}{d-2s}R^2+|x|^2\right)dx \cr
&= \frac{\lambda K_{d,s}}{4}\frac{d\pi^{d/2}(d+2-2s)\Gamma(2-s)}
{(d-2s)\Gamma(d/2+3-s)}R^{d+4-2s}
=  \tilde{K}_{d,s} \left(\int_{\mathbb{R}^d} \rho(x)dx\right)^{\frac{d+4-2s}{d+2-2s}},
\end{align*}
where~\eqref{eq:MR} is used in the last step, together with the constant
\[
\tilde{K}_{d,s}=\frac{d(d+2-2s)^{(d+4-2s)/(d+2-2s)}\lambda^{(d-2s)/(d+2-2s)}}
{(d-2s)(d+4-2s)2^{(d+2-s)/(d+2-2s)}\pi^{d/(d-2-2s)}}.
\]

Therefore,  ~\eqref{LSI} is reduced to an inequality bounding the integral $\int \rho (-\Delta)^{-s}\rho\,dx$
by $\int \rho\,dx$ and $\int \rho |\nabla (-\Delta)^{-s}\rho|^2\,dx$, that is,
\[
\lambda(d+1-2s)\int_{\mathbb{R}^d}\rho(x)(-\Delta)^{-s} \rho(x)dx
\leq 2\lambda \tilde{K}_{d,s} \left(\int_{\mathbb{R}^d} \rho(x)dx\right)^{\frac{d+4-2s}{d+2-2s}}+
\int_{\mathbb{R}^d} \rho(x)|\nabla(-\Delta)^{-s}\rho(x)|^2dx,
\]
where the equality holds for the steady state $\rho_\infty$. In general, it is easier to prove the equivalent inequality in the
``product form''
\begin{equation}\label{eq:fracGNS}
\int_{\mathbb{R}^d} \rho (-\Delta)^{-s}\rho\,dx
\leq C\left(\int_\Rd \rho \,dx\right)^{2-3\theta}
\left( \int_\Rd \rho|\nabla(-\Delta)^{-s}\rho|^2 \, dx\right)^{\theta},
%\qquad
\end{equation}
where $\theta =  \frac{d-2s}{2d+2-4s}$ is determined by the homogeneity and $C$ is given
by  any function $\rho(x) = A(R^2-|x-x_0|^2)_+^{1-s}$ (which is independent of $A$, $R$ and $x_0$).

However, unlike the case of porous medium equation~\cite{dolbeault},
we can not prove~\eqref{eq:fracGNS} to
establish the log-Sobolev inequality~\eqref{LSI}. The main difficulty lies in the integral $\int_\Rd \rho|\nabla(-\Delta)^{-s}\rho|^2$,
where basic questions like monotonicity under symmetric decreasing rearrangement are not clear.
Because of the equivalence between~\eqref{LSI} and~\eqref{eq:fracGNS}, we will show that~\eqref{eq:fracGNS} holds in one dimension and it is a consequence of the HWI inequalities, but it remains an open problem to prove or disprove~\eqref{eq:fracGNS} in higher dimensions.

To summarize, provided the required regularity of the solutions in the formal calculation in manipulating $\R(\rho)$, the exponential convergence
of solutions to~\eqref{eq:fracPMESM} is expected only in one dimension,
which the equivalent inequality~\eqref{eq:fracGNS} can not be
proved at this moment. The convergence in one dimension will be
established more rigorously in the next two sections, by showing
an even more general HWI inequality related the displacement convexity
of the energy.

%===============================================================================================================================

%=======================================================SECTION 3===============================================================

\section{Transport inequalities}\label{TranspIneq}

In this section, we derive several inequalities originated from
optimal transportation theory that will be used in the next section to show
the exponential convergence of the relative
entropy in one dimension. Besides  $\mathcal{E}(\rho)$
and  $\mathcal{I}(\rho)$ introduced earlier, we also
need the following versions of the energy and energy dissipation of a measure $\rho\in \mathcal{P}_{2,ac}(\R)$:
\begin{align*}
\E_\varepsilon(\rho) &:= \E(\rho) +\varepsilon \int_{\mathbb{R}} \rho \log \rho\,, \\
\I_\varepsilon(\rho) &:= \int_{\mathbb{R}} \left| \partial_x(-\partial_{xx})^{-s}\rho(x)+ \lambda x +\varepsilon\partial_x\log \rho(x)  \right|^2 d\rho(x)\,,
\end{align*}
which are associated to the regularized equation~\eqref{FPME} in the next section. Throughout this and the next sections we shall commit an abuse of notation and identify every absolutely continuous measure with its density. So we shall write $d\rho(x)$ and $\rho(x)dx$ meaning the same thing.

We use optimal transport techniques to prove the Log-Sobolev, the Talagrand, and the HWI inequalities for the energy $\E_\varepsilon$ for smooth probability measures $\rho \in \mathcal{P}_{2,ac}(\R)$.  We shall focus on the so called HWI inequality that generalizes
certain elementary inequalities for convex functions on $\mathbb{R}^d$ with
Euclidean distance replaced by the Wasserstein distance on $\mathcal{P}_2(\R)$ (the space of probability measures with finite second moment).
The Wasserstein distance on $\mathcal{P}_2(\R)$ is defined for any $\rho_1, \rho_2 \in \mathcal{P}_2(\R)$ by
\[
W_2(\rho_1, \rho_2) := \left( \inf_{\pi \in \Pi(\rho_1,\rho_2)} \int_{\R\times\R} |x-y|^2\;d\pi(x,y)  \right)^{\frac{1}{2}}\;,
\]
where $\Pi(\rho_1,\rho_2)$ be the set of all nonnegative Radon measures on $\R\times\R$ with marginals (projections) $\rho_1$ and $\rho_2$.
The HWI inequality is called so because it was first established in~\cite{ottovillani} for the relative Kullback information (denoted by $H$), the Wasserstein distance $W_2$ and the relative Fisher information (also denoted by $I$).

Before stating the main results, let us briefly review a few facts about the Wasserstein distance and the weak convergence in $\mathcal{P}_2(\R)$ that shall be used in the proofs.
\begin{itemize}
  \item We say that the a sequence $(\rho_n)_{n\in\mathbb{N}}\subseteq \mathcal{P}_2(\R)$ weakly converges to $\rho\in\mathcal{P}(\R)$ (denoted as $\rho_n \rightharpoonup \rho$), if
  \[\lim_{n\to\infty}\int\varphi(x) \;d\rho_n(x) = \int\varphi(x) \;d\rho(x)\;,\;\; \]
  for all $\varphi \in C_b(\R)$, the space of bounded and continuous functions.
  \item The pair $(\mathcal{P}_2(\R),W_2)$ is a complete metric space and the convergence under the distance $W_2$ is stronger than the convergence in the weak sense. In fact, the following facts are equivalent for any $(\rho_n)_{n\in\mathbb{N}}\subseteq \mathcal{P}_2(\R)$ and $\rho \in \mathcal{P}(\R)$:
      \begin{itemize}
      \item[i)] $W_2(\rho_n,\rho)\rightarrow 0$ as $n \rightarrow +\infty$;
      \item[ii)] $\rho_n \rightharpoonup \rho$ and
      \begin{equation}\label{limsup}
      \lim_{n\to\infty}\int x^2\;d\rho_n(x) = \int x^2\;d\rho(x);
      \end{equation}
      \item[iii)] $\rho_n \rightharpoonup \rho$ and
      \begin{equation*}
      \lim_{R\to\infty} \limsup_{n\to\infty} \int_{|x|\maig R} x^2\;d\rho_n(x) = 0.
      \end{equation*}
      \end{itemize}
  \item Given $\rho_1, \rho_2 \in \mathcal{P}_2(\R)$ with $\rho_1$ absolutely continuous with respect to the Lebesgue measure, there exists a Borel map $\theta:\R \rightarrow \R$ such that $\theta\#\rho_1 = \rho_2$, i.e.,
      \[\int_\R \varphi(x) \;d\rho_2(x) = \int_\R \varphi(\theta(x))\;d\rho_1(x), \mbox{ for every bounded Borel function } \varphi,\]
      and $\theta$ also satisfies
      \[W_2(\rho_1,\rho_2) = \left( \int_\R |x-\theta(x)|^2\;d\rho_1(x)  \right)^{\frac{1}{2}}\;,\]
It is well known that the optimal map $\theta$ is nondecreasing on $\R$ and increasing on $\supp(\rho_1)$.
\end{itemize}

For a detailed proof of the above results and generalizations, the reader may check the standard references~\cite{ambrosiogiglisavare} and~\cite{villani}. Now, let us begin with the following technical lemma about the gradient of the
Riesz potential in general dimension $d$.

\begin{lemma}
Let $0<s\meni 1$ and $\rho \in L^1(\R^d)\cap L^\infty(\R^d)\cap C^\alpha(\R^d)$ with $\alpha>\max(1-2s,0)$. Then $(-\Delta)^{-s}\rho \in C^1(\R^d)$ and for any $x \in \R^d$,
\[ \nabla(-\Delta)^{-s}\rho(x) =  -c_{d,s}(d-2s) \int_{\R^d} \frac{x-y}{|x-y|^{d+2-2s}}\Big(\rho(y)-\rho(x) \Big)\;dy \;,\;\; \mbox{ if } s \in (0, 1/2] \]
or
\[ \nabla(-\Delta)^{-s}\rho(x) =  -c_{d,s}(d-2s) \int_{\R^d} \frac{x-y}{|x-y|^{d+2-2s}}\rho(y)\;dy \;,\;\; \mbox{ if } s\in (1/2,1]. \]
\end{lemma}

\begin{proof}
For $s=1$ and $d\maig 2$ this result is in \cite[Lemma 4.1]{MR1814364} for the Newtonian potential, and one only needs $\rho \in L^\infty(\R^d)\cap L^1(\R^d)$ in order to have $(-\Delta)^{-1}\rho \in C^1(\R^d)$ with
\[\nabla(-\Delta)^{-1}\rho(x) =  -c_{d,1}(d-2) \int_{\R^d} \frac{x-y}{|x-y|^{d}}\rho(y)\;dy. \]
So let us assume that $s \in (0,1/2]$ if $d \maig 2$ and $s \in (0,1/2)$ if $d =1$. To simplify the notation, we write $k_{d,s}(x) := c_{d,s}|x|^{2s-d}$. Hence, we note that under the hypothesis on $\rho$, we have that
\[\textbf{u}_{d,s}(x) := -c_{d,s}(d-2s) \int_{\R^d} \frac{(x-y)}{|x-y|^{d+2-2s}}\Big(\rho(y)-\rho(x) \Big)\;dy = \nabla k_{d,s} \ast (\rho - \rho(x))\]
is well defined for all $x \in \R^d$.

Now, let $\eta \in C^1(\R^d)$ be a radial function such that $0\meni \eta \meni 1$, $\eta(x) = 0$ if $|x|\meni 1$, $\eta(x) = 1$ if $|x|\maig 2$ and $|\nabla\eta|\meni 2$. Define $\eta_\varepsilon(x) := \eta (\varepsilon^{-1}x)$ and
\begin{align*}
p(x)&:= (-\Delta)^{-s}\rho(x) = k_s\ast\rho (x) \\
p_\varepsilon (x) &:= (k_{d,s}\eta_\varepsilon)\ast\rho (x)
\end{align*}
Since $\rho$ is bounded, we have that $p \rightarrow p_\varepsilon$ uniformly on $\R^d$ as
\begin{align*}
|p(x)-p_\varepsilon(x)| &\meni \int_{|x-y|\meni 2\varepsilon} k_{d,s}(x-y)\big(1-\eta_\varepsilon(x-y)\big)\rho(y)\;dy \\
&\meni \|\rho\|_\infty \int_{|y|\meni 2\varepsilon}\frac{1}{|y|^{d-2s}}\;dy = C\|\rho\|_\infty \varepsilon^{2s}
\end{align*}
for all $x \in \R^d$, where $C$ depends on $d$ and $s$.

By the smoothness of $k_{d,s}\eta_\varepsilon$ we know that $p_\varepsilon \in C^1$ and $\nabla p_\varepsilon (x) = \nabla(k_{d,s}\eta_\varepsilon)\ast\rho (x)$, and since $k_{d,s}\eta_\varepsilon$ is radial, we can write
\[\nabla p_\varepsilon (x) = \int_{\mathbb{R}^d} \nabla(k_{d,s}\eta_\varepsilon)(x-y)\Big(\rho(y)-\rho(x)\Big)\;dy\,.
\]
Therefore,
\begin{align}
|\textbf{u}_{d,s}(x) - \nabla p_\varepsilon (x)| &= \left| \int_{|x-y|\meni 2\varepsilon} \nabla(k_{d,s} (1-\eta_\varepsilon) )(x-y)\Big(\rho(y)-\rho(x) \Big)\;dy \right| \nonumber \\
&\meni \int_{|x-y|\meni 2\varepsilon} \Big(|\nabla k_{d,s}(x-y)||1-\eta_\varepsilon(x-y)|+ k_{d,s}(x-y)|\nabla\eta_\varepsilon(x-y)| \Big)\Big|\rho(y)-\rho(x) \Big|\;dy \nonumber \\
&\meni \int_{|x-y|\meni 2\varepsilon} \Big(\frac{c_{d,s}(d-2s)}{|x-y|^{d+1-2s}}+ \frac{2}{\varepsilon}\frac{c_{d,s}}{|x-y|^{d-2s}} \Big)\Big|\rho(y)-\rho(x) \Big|\;dy \label{ApIne}\\
&\meni C \int_{|x-y|\meni 2\varepsilon} \left(\frac{1}{|x-y|^{d+1-2s-\alpha}} + \frac{1}{\varepsilon}\frac{1}{|x-y|^{d-2s-\alpha}}\right)\;dy\nonumber\\
&\meni C_1\varepsilon^{\alpha+2s-1},\nonumber
\end{align}
where the constant $C_1$ only depends on $d$, $s$, $\alpha$ and on the H\"{o}lder constant of $\rho$. Thus, we also have that $\nabla p_\varepsilon$ converges uniformly to $\textbf{u}_{d,s}$ as $\varepsilon\rightarrow 0$, and therefore $\nabla p = \textbf{u}_{d,s}$.

Now, if $s \in (1/2,1)$ and $d \maig 2$ or $s \in (1/2,1]$ and $d =1$, we only need to adapt the argument in formula~\eqref{ApIne} for the function
\[\textbf{u}_{d,s}(x) := -c_{d,s}(d-2s) \int_{\R^d} \frac{x-y}{|x-y|^{d+2-2s}}\rho(y)\;dy = \nabla k_{d,s} \ast \rho \]
and using that $\nabla p_\varepsilon = \nabla (k_{d,s}\eta_\varepsilon)\ast \rho$ in the following way
\begin{align*}
|\textbf{u}_{d,s}(x) - \nabla p_\varepsilon (x)| &= C\norm{\rho}_\infty \int_{|x-y|\meni 2\varepsilon} \left(\frac{1}{|x-y|^{d+1-2s}} + \frac{1}{\varepsilon}\frac{1}{|x-y|^{d-2s}}\right)\;dy\\
&= C_2\varepsilon^{2s-1},
\end{align*}
where the constant $C_2$ only depends on $d$, $s$ and on the $L^\infty$ norm of $\rho$.

Finally, if $d=1$ and $s=1/2$ we have that
\[(-\partial_{xx})^{-\frac{1}{2}}\rho(x) = -c_{1,\frac{1}{2}}\int_\R \log|x-y|\rho(y)\;dy\]
and
\[ \textbf{u}_{1,\frac{1}{2}}(x) = c_{1,\frac{1}{2}} \int_{\R^d} \frac{(x-y)}{|x-y|^{2}}\Big(\rho(y)-\rho(x) \Big)\;dy.  \]
Arguing as above for $k_{1,\frac{1}{2}}(x) := -c_{1,\frac{1}{2}}\log|x|$ we arrive at the following estimates:
\begin{align*}
|p(x)-p_\varepsilon(x)| &\meni \|\rho\|_\infty \int_{|y|\meni 2\varepsilon}\big|\log|y|\big|dy = C\|\rho\|_\infty \varepsilon\big(\big|\log{2\varepsilon}\big| +1\big)
\end{align*}
and
\begin{align*}
|\textbf{u}_{1,\frac{1}{2}}(x) -  p'_\varepsilon (x)| &\meni C \int_{|x-y|\meni 2\varepsilon} \Big(\frac{1}{|x-y|}+ \frac{1}{\varepsilon}\big|\log|x-y|\big| \Big)\Big|\rho(y)-\rho(x) \Big|\;dy \\
&\meni C \int_{|x-y|\meni 2\varepsilon} \Big(\frac{1}{|x-y|^{1-\alpha}}+ \frac{1}{\varepsilon}|x-y|^{\alpha}\big|\log|x-y|\big| \Big)\;dy \\
&\meni C\varepsilon^{\alpha}\Big(1+\varepsilon + \varepsilon\big|\log{2\varepsilon}\big| \Big).
\end{align*}
Therefore, since all these estimates are uniform in $x$, we conclude that the lemma is true for all $s \in (0,1]$ and $d\maig 1$.
\end{proof}

\begin{remark}\label{remark1}
{\rm With this expression for the derivative of $(-\Delta)^{-s}\rho$ for $s < \frac{1}{2} $, we obtain the following equality that shall be used in the next proposition:
\begin{align*}
\frac{\nabla(-\Delta)^{-s}\rho(x)}{c_{d,s}(2s-d)} &=  \lim_{r\to 0}\int_{|x-y|\maig r} \frac{x-y}{|x-y|^{d+2-2s}}\Big(\rho(y)-\rho(x) \Big)\;dy \\
&=  \lim_{r\to 0}\int_{|x-y|\maig r} \frac{x-y}{|x-y|^{d+2-2s}}\rho(y)\;dy- \lim_{r\to 0}\rho(x)\int_{|x-y|\maig r}\frac{x-y}{|x-y|^{d+2-2s}}\;dy\\
&= \lim_{r\to 0}\int_{|x-y|\maig r} \frac{x-y}{|x-y|^{d+2-2s}}\rho(y)\;dy,
\end{align*}
where we only used the fact that $k_s$ is radial and $\nabla k_s$ is integrable at the infinity. For $s
> \frac{1}{2}$, the expression is valid without taking the limit, as the kernel is locally integrable.}
\end{remark}

The next proposition shows that the HWI inequality holds for $\E$ and $\E_\varepsilon$ at least for a class of bounded and H\"older continuous functions on $\R$. The proof follows the arguments given in~\cite{ledoux} where the same inequality is proved for the case of the logarithmic interaction and strongly relies on the fact that the optimal transport map w.r.t the Wasserstein distance is a monotone nondecreasing function on $\R$.  We point out that the convexity of the confinement due to the drift measured by $\lambda>0$ appears explicitly in the inequalities as in \cite{MR2053570}.

\begin{theorem}\label{HWI}
Let $s\in (0,1]$, $\lambda \in \R$, $\rho\in L^1(\R)\cap L^\infty(\R)\cap C^\alpha(\R)$ nonnegative where $\alpha >\max(1-2s,0)$ and with $\int\rho = 1$, and $\rho_\infty$ the minimum point of $\E$ on $\mathcal{P}^{2}(\R)$. Then
\[\E(\rho)-\E(\rho_\infty) \meni \sqrt{\I(\rho)}W_2(\rho,\rho_\infty) - \frac{\lambda}{2} W_2^2(\rho,\rho_\infty).\]
\end{theorem}

\begin{proof}
For $s=1/2$ this result was proven at~\cite{ledoux}. So, let us suppose that $s\in(0,1/2)$ and, to simplify, let us denote $K\rho(x) = \partial_x(-\partial_{xx})^{-s}\rho(x)$.
Since $\rho$ is absolutely continuous with respect to the Lebesgue measure, there exists an nondecreasing transport map $\theta$ such that $\theta\#\rho=\rho_\infty$.

Then, let us write
\[\sqrt{\I(\rho)}W_2(\rho,\rho_\infty) - \frac{\lambda}{2} W_2^2(\rho,\rho_\infty) -\E(\rho)+\E(\rho_\infty)  = T_1 + T_2 + T_3\]
where
\begin{align*}
T_1 := &\left( \int \Big| K\rho(x)+ \lambda x  \Big|^2 d\rho(x) \right)^{1/2}\left(\int|x-\theta(x)|^2d\rho(x)\right)^{1/2} \\ &-\int\Big(K\rho(x)+\lambda x\Big)\left(x-\theta(x)\right)d\rho(x)
\end{align*}
\begin{equation*}
T_2 := \int \Big\{\lambda x(x-\theta(x)) - \frac{\lambda}{2}x^2 + \frac{\lambda}{2}\theta(x)^2 -\frac{\lambda}{2}|x-\theta(x)|^2 \Big\}d\rho(x)
\end{equation*}
\[T_3 :=  \frac{c_{1,s}}{2}\int\frac{d\rho(x)d\rho(y)}{|\theta(x)-\theta(y)|^{1-2s}} - \frac{c_{1,s}}{2}\int\frac{d\rho(x)d\rho(y)}{|x-y|^{1-2s}} - \int K\rho(x)(\theta(x)-x)d\rho(x) \,,
\]
 where we added and subtracted several terms. This allows us to show that $T_1\maig 0$ by the Cauchy-Schwarz inequality and $T_2 = 0$ for all $\lambda \in \R$. Now, for $T_3$ let us call $k_s(x) = c_{1,s}|x|^{2s-1}$. Then, by the Remark~\ref{remark1}
\[K\rho(x) = \lim_{r\to 0}\int_{|y-x|\maig r}k_s'(x-y)d\rho(y)  \]
And, since $k_s'(x) = -k_s'(-x)$, we can write
\begin{align*}
\int K\rho(x)&\big(\theta(x)-x\big)d\rho(x) \\
&= \lim_{r\to 0}\int_{|y-x|\maig r} \big(\theta(x)-x\big)k_s'\big(x-y\big)d\rho(y)d\rho(x)\\
&= \frac{1}{2}\lim_{r\to 0}\int_{|y-x|\maig r} \big(\theta(x)-\theta(y)-x+y)k_s'\big(x-y\big)d\rho(y)d\rho(x)
\end{align*}
Furthermore,
\[c_{1,s}\int \frac{d\rho(x)d\rho(y)}{|x-y|^{1-2s}} = \lim_{r\to 0}\int_{|y-x|\maig r} k_s(x-y)d\rho(x)d\rho(y) \]
\[c_{1,s}\int \frac{d\rho(x)d\rho(y)}{|\theta(x)-\theta(y)|^{1-2s}} = \lim_{r\to 0}\int_{|y-x|\maig r} k_s(\theta(x)-\theta(y))d\rho(x)d\rho(y) \]
and then,
\[
T_3 = \lim_{r\to 0}\frac{1}{2} \int \Big\{k_s\big(\theta(x)-\theta(y)\big) -  k_s(x-y) - k_s'\big(\theta(x)-\theta(y)\big)\big(\theta(x)-\theta(y)-x+y\big)    \Big\}d\rho(x)d\rho(y)
\]
The integrand is nonnegative by the convexity of $k_s$ on the positive real line and by the monotonicity of $\theta$, so $T_3\maig 0$ as well.

If $s \in (1/2,1]$, we still have $k_s(x) = c_{1,s}|x|^{2s-1}$ convex because $c_{1,s}$ is negative in this range. Thus, the previous computations  still apply.
\end{proof}

\noindent {\bf Remarks.} 1) It is known that, if the HWI inequality holds for some $\lambda >0$, then the Log-Sobolev inequality also holds. One just needs to maximize the right-hand side for $W_2 \maig 0$ or use the Young's inequality for $(\lambda^{-\frac{1}{2}} \sqrt{\I}) (\lambda^\frac{1}{2}W_2)$. Then we have that
\begin{equation}\label{LogSobolev}
\E(\rho)-\E(\rho_\infty) \meni \frac{1}{2\lambda}\I(\rho),
\end{equation}
for all $\rho$ satisfying the assumptions of the theorem above.

2) Note that in the proof of the Theorem~\ref{HWI} we did not use the fact that $\rho_\infty$ is the minimum of $\E$, only its regularity. In fact, the same inequality holds for any $\rho_0$ in the place of $\rho_\infty$, and also with $\rho_\infty$ in the place of $\rho$, because $\rho_\infty \in L^\infty(\R)\cap C^{1-s}(\R)$, which allows the existence of $\theta$. Therefore, if we exchange $\rho$ and $\rho_\infty$ in the HWI we obtain the fractional version of the so called Talagrand inequality or transportation cost inequality
\begin{equation}\label{Talagrand}
W_2(\rho,\rho_\infty) \meni \sqrt{\frac{2}{\lambda}\Big( \E(\rho)-\E(\rho_\infty) \Big). }
\end{equation}

We can derive similar results for the $\varepsilon$ problems.

\begin{proposition}\label{HWI-e}
Let $s \in (0,1]$, $\lambda >0$, $0<\varepsilon<\lambda/2\pi$, $\rho\in L^1(\R)\cap L^\infty(\R)\cap C^\alpha(\R)$ nonnegative where $\alpha >1-2s$ and with $\int\rho = 1$, and $\rho_\infty^\varepsilon$ the minimum point of $\E_\varepsilon$ on $\mathcal{P}_{2}(\R)$. Then
\[\E_\varepsilon(\rho)-\E_\varepsilon(\rho_\infty^\varepsilon) \meni \sqrt{\I_\varepsilon(\rho)}W_2(\rho,\rho^\varepsilon_\infty) - \frac{\lambda}{2} W_2^2(\rho,\rho_\infty^\varepsilon)\]
\end{proposition}

\begin{proof}
The proof is basically the same, but since we have a new term inside the respective diffusion, we shall include it for completeness.

As in the previous theorem, let $K\rho(x) = \partial_x(-\partial_{xx})^{-s}\rho(x)$ and $\theta$ be such that $\theta\#\rho=\rho_\infty^\varepsilon$. Then, we decompose the inequality as
\[\sqrt{\I_\varepsilon(\rho)}W_2(\rho,\rho_\infty^\varepsilon) - \frac{\lambda}{2} W_2^2(\rho,\rho_\infty^\varepsilon) -\E_\varepsilon(\rho)+\E_\varepsilon(\rho_\infty^\varepsilon)  = T_1 + T_2 + T_3\]
where
\begin{align*}
T_1 := &\left( \int \Big| K\rho(x)+ \lambda x +\varepsilon\partial_x\log\rho(x) \Big|^2 d\rho(x) \right)^{1/2}\left(\int|x-\theta(x)|^2d\rho(x)\right)^{1/2} \\ &-\int\Big(K\rho(x)+\lambda x+\varepsilon\partial_x\log\rho(x)\Big)\left(x-\theta(x)\right)d\rho(x)
\end{align*}
\begin{align*}
T_2 := -&\int \Big(\varepsilon\partial_x\log\rho(x) + \lambda x\Big)(\theta(x)-x)\;d\rho  - \int \Big( \frac{\lambda}{2}x^2 +\varepsilon\log\rho\Big)\;d\rho \\
    &+\int \Big( \frac{\lambda}{2}x^2 +\varepsilon\log\rho_\infty^\varepsilon\Big)\;d\rho_\infty^\varepsilon -\frac{\lambda}{2} \int |x-\theta(x)|^2d\rho(x)
\end{align*}
\[T_3 :=  \frac{c_{1,s}}{2}\int\frac{d\rho(x)d\rho(y)}{|\theta(x)-\theta(y)|^{1-2s}} - \frac{c_{1,s}}{2}\int\frac{d\rho(x)d\rho(y)}{|x-y|^{1-2s}} - \int K\rho(x)(\theta(x)-x)d\rho(x) \]

By the same arguments, we conclude that $T_1,T_3 \maig 0$. Now, for $T_2$, let us define the following functional
\[H(f|g):= \int f(x)\log\left(\frac{f(x)}{g(x)}\right)\;dx\]
for all $f,g\in L^1(\R)$. Then we can re-write $T_2$ in the following way
\begin{align*}
T_2 = &\varepsilon\left( -\int \partial_x\log\left(\frac{\rho(x)}{e^{-\pi x^2}}\right)(\theta(x)-x)\;d\rho(x) - H(\rho|e^{-\pi x^2})+ H(\rho_\infty^\varepsilon|e^{-\pi x^2})  + \pi\int |\theta(x)-x|^2\;d\rho  \right)\\
&+\left(1-\frac{2\pi}{\lambda}\varepsilon\right)\int\left\{ -\lambda x (\theta(x)-x)-\frac{\lambda}{2}x^2+\frac{\lambda}{2}\theta(x)^2+\frac{\lambda}{2}(\theta(x)-x)^2      \right\}\;d\rho(x).
\end{align*}

Note that the second line is equal to $(\lambda -2\pi\varepsilon)\int |\theta(x)-x|^2\;dx$, which is nonnegative for $\varepsilon<\lambda/2\pi$. For the first line, we can use the proof of the HWI inequality made in \cite{ottovillani}. Actually, Otto and Villani showed that whenever $f,f_0\in C_c^{\infty}(\R)\cap\mathcal{P}(\R)$ and $V\in C^2(\R)$ is such that $\int e^{-V}dx = 1$ and $V''\maig K$ for some $K\in\R$, then
\[H(f_0|e^{-V}) - H(f|e^{-V}) - \int \partial_x\log\frac{f(x)}{e^{-V(x)}}(\theta(x)-x)f(x)\;dx - \frac{K}{2}\int|\theta(x)-x|^2f(x)\;dx\;\maig 0,\]
and for the density argument given in the proof of the Theorem 9.17 of \cite{villani}, we have that this inequality holds for all $f,f_0 \in L^1(\R)\mathcal{P}_{2}(\R)$. So, applying this for $V(x)=\pi x^2$ we have that $K=2\pi$ and we conclude that $T_2\maig 0$.
\end{proof}

\begin{remark}
By the same arguments given for~\eqref{LogSobolev} and ~\eqref{Talagrand}, we conclude that the following Log-Sobolev and Talagrand inequalities hold for $\E_\varepsilon$, as long as $\rho$ satisfies the assumptions of proposition~\ref{HWI-e}:
\begin{equation}\label{LogSobolev-e}
\E_\varepsilon(\rho)-\E_\varepsilon(\rho^\varepsilon_\infty) \meni \frac{1}{2\lambda}\I_\varepsilon(\rho),
\end{equation}
\begin{equation*}%\label{Talagrand-e}
W_2(\rho,\rho^\varepsilon_\infty) \meni \sqrt{\frac{2}{\lambda}\Big( \E_\varepsilon(\rho)-\E_\varepsilon(\rho^\varepsilon_\infty) \Big). }
\end{equation*}
\end{remark}

\begin{remark}
 These results also work for a general confinement potential $V:\R\rightarrow \R$ instead of the quadratic one $\frac{\lambda}{2}x^2$, as long as $V-\frac{\lambda}{2}x^2$ is convex.
\end{remark}

Finally, let us prove the following lemma that shall be used in the last section for the convergence in entropy of the solutions of the approximate problems. The proof uses similar arguments given in the Theorem 1.4 of ~\cite{MR1485778}. Let us just remind that a sequence $\{\rho_n\}_{n\in \mathbb{N}}\subseteq \mathcal{P}(\R)$ is said to converge in the weak-$\ast$ sense to $\rho \in \mathcal{P}(\R)$, $\rho_n \stackrel{\ast}{\rightharpoonup}  \rho$ if
\[\lim_{n\to \infty}\int_\R \varphi(x) d\rho_n(x) =  \int_\R \varphi(x) d\rho(x)\;,\;\; \mbox{ for all } \varphi \in C_0(\R)\]
where $C_0(\R)$ is the space of continuous functions on $\R$ that goes to zero at infinity. It is clear that convergence in $W_2$ implies weak convergence and weak convergence implies weak-$\ast$ convergence.

\begin{lemma}
The entropy $\E_\varepsilon$ is weak-$\ast$ lower semi-continuous for all $\varepsilon\maig 0$.
\end{lemma}

\begin{proof}
We know from~\cite{MR1451422} that the functional
\[\rho \mapsto \int \rho \log \rho\]
is weak-$\ast$ lower semi-continuous, so we just need to show the result for $\E$. For this, let us write it in the following way:
\[\E(\rho) = \int_{\R^2} F(x,y) d\rho(x)d\rho(y),\]
where
\[F(x,y) = \displaystyle\left\{ \begin{array}{cc}\displaystyle  \frac{\lambda}{4}(x^2+y^4) + \frac{c_{1,s}}{2}\frac{1}{|x-y|^{1-2s}}\;, & \mbox{ if } x\neq y \\ +\infty\hspace{3.9cm}, & \mbox{ if } x= y \end{array}\right..\]

Since $F$ is non-negative and smooth outside the diagonal $x=y$, we can find a sequence $\{F_k\}_{k \in \mathbb{N}}\subset C_0(\R^2)$ such that $F_k(x,y)\nearrow F(x,y)$ for all $(x,y)\in \R^2$. Therefore, by the monotone convergence theorem and the fact that $\rho_n\times\rho_n \stackrel{\ast}{\rightharpoonup} \rho\times\rho $ if $\rho_n \stackrel{\ast}{\rightharpoonup} \rho$, we have that
\begin{align*}
\E(\rho) &= \int F(x,y)\;d\rho(x)d\rho(y) = \lim_{k\to\infty}\int F_k(x,y)\;d\rho(x)d\rho(y) \\
&= \lim_{k\to\infty} \lim_{n\to\infty}\int F_k(x,y)\;d\rho_n(x)d\rho_n(y) \meni  \lim_{n\to\infty}\int F(x,y)\;d\rho_n(x)d\rho_n(y) \\
&= \liminf_{n\to\infty} \E(\rho_n)\,.
\end{align*}
\end{proof}

%===================================================================================================
%============================================SECTION 4==============================================
\section{Exponential Convergence}

In Section~\ref{BE}, most of the calculations are performed at a formal level, assuming some strong regularity on the solutions of~\eqref{eq:fracPMESM} that has not been proved at the moment (see \cite{vazquez3} for the proof of H\"older regularity). In this section, to avoid this regularity issues, we shall prove that the energy of the solution decays exponentially fast for the regularized equation with mollified initial data, and then passing the limit on these regularizing parameters.

\begin{theorem}\label{Theorem}
Let $\rho_0 \in L^1(\R)\cap L^{\infty}(\R)$ such that
\[0\meni \rho_0(x) \meni Ae^{-a|x|}\;,\;\;\]
for some constants $a,A > 0$.
Then, for each $0<s<1/2$, the solution $\rho(t,\cdot)$ of ~\eqref{eq:fracPMESM} with initial data $\rho_0$ satisfies
\[\E(\rho(t))-\E(\rho_\infty) \meni e^{-2\lambda t}\Big(\E(\rho_0)-\E(\rho_\infty) \Big).\]
\end{theorem}

\begin{proof}
In order to use the results of Section \ref{TranspIneq}, firstly we shall assume that
\begin{equation}\label{InitialRegularity}
\rho_0 \in C^\infty(\R) \;\; \mbox{ and } \;\; \int_\R \rho_0(x) \;dx =1.
\end{equation}
Let $\rho_\infty,\rho^\varepsilon_\infty \in \mathcal{P}(\R)$ be the minimizers for $\E$ and $\E_\varepsilon$ respectively. By the assumption on $\rho_0$ we know from the proofs of Theorems 4.1 and 4.2
in~\cite{vazquez2} that the solutions $\rho$ and $\rho^\varepsilon$
to
\begin{equation} \label{FPME}
\left\{ \begin{array}{cc}
             \partial_t\rho = \partial_x(\rho \partial_x (-\partial_{xx})^{-s}\rho + \lambda x\rho)  &,  \mbox{ in } \R\times(0, \infty) \\
             \rho(0) = \rho_0 \hspace{3.6cm} &  , \mbox{ in } \R, \hspace{1.3cm}
           \end{array}
 \right.
\end{equation}
and
\begin{equation} \label{FPMEd}
\left\{ \begin{array}{ccc}
             \partial_t\rho^\varepsilon = \partial_x(\rho^\varepsilon \partial_x (-\partial_{xx})^{-s}\rho^\varepsilon +\lambda x \rho^\varepsilon ) + \varepsilon\partial_{xx}\rho^\varepsilon  &, \mbox{ in } \R\times(0, \infty) \\
             \rho^\varepsilon(0) = \rho_0 \hspace{5.5cm} & , \mbox{ in } \R \hspace{1.3cm}
           \end{array}
 \right.
\end{equation}
satisfy $\rho \in C([0,\infty);L^1(\R))$ and $\rho^\varepsilon \in C^1((0,\infty)\times \R)$
for all $\varepsilon >0$ sufficiently small.
Because of the regularization in~\eqref{FPMEd}, for fixed time $t>0$, $\rho^\varepsilon(t,.)$ is in fact in $C^2(\R)$.
Moreover, there exist  $C(t),a(t) > 0$, such that
\begin{equation}\label{ExpBound}
0 \meni \; \rho(t,x)\;,\;\rho^\varepsilon(t,x) \;\meni\; C(t)e^{-a(t)|x|}.
\end{equation}

Since $\rho^\varepsilon(t)$ is smooth, we can apply the Log-Sobolev Inequality~\eqref{LogSobolev-e} for $\E_\varepsilon$ and obtain that for all $t\maig 0$,
\[\E_\varepsilon(\rho^\varepsilon(t)) - \E_\varepsilon(\rho^\varepsilon_\infty) \meni \frac{1}{2\lambda}\I_\varepsilon(\rho^\varepsilon(t)).\]
Making use of the fact that
\[\frac{d}{dt}\E_\varepsilon(\rho^\varepsilon(t)) = -\I_\varepsilon(\rho^\varepsilon(t)),\]
we conclude that
\begin{equation}\label{eq:decay-e}
\E_\varepsilon(\rho^\varepsilon(t)) - \E_\varepsilon(\rho^\varepsilon_\infty) \meni e^{-2\lambda t}\Big(\E_\varepsilon(\rho_0) - \E_\varepsilon(\rho^\varepsilon_\infty) \Big).
\end{equation}

To take the limits as $\varepsilon \rightarrow 0^+$,  let us analyze each term
on both sides of~\eqref{eq:decay-e} separately:
\begin{itemize}
\item[i)] The easiest one is the limit $\E_\varepsilon(\rho_0)$, since $\displaystyle\lim_{\varepsilon \to 0^+}\E_\varepsilon(\rho_0)=\E(\rho_0)$ holds as long as $\E_\varepsilon(\rho_0)<\infty$ for some $\varepsilon
> 0$, which is true by the assumptions on $\rho_0$.

%=====================================================================================
\item[ii)] For the term $\E_\varepsilon(\rho^\varepsilon_\infty)$, let us first define the following
auxiliary functional on $\mathcal{P}_{2,ac}(\R)$:
\[ \mathcal{H}(\rho) := \mathcal{H}(\rho|e^{-\pi x^2}) = \pi\int x^2 \rho +\int \rho \log\rho.\]
Since $\int e^{-\pi x^2}dx = 1$, we can write
\[\mathcal{H}(\rho) = \int \frac{\rho}{e^{-\pi x^2}}\log\left(\frac{\rho}{e^{-\pi x^2}}\right)e^{-\pi x^2}\;dx = \int \left[ \frac{\rho}{e^{-\pi x^2}}\log\left(\frac{\rho}{e^{-\pi x^2}}\right) - \frac{\rho}{e^{-\pi x^2}} +1  \right] e^{-\pi x^2}\;dx,\]
which is nonnegative by Jensen's inequality.

Let us prove that $\limsup_{\varepsilon \to 0}\E_\varepsilon(\rho_\infty^\varepsilon)\meni \E(\rho_\infty)$. Using the fact that $\rho^\varepsilon_\infty$ is the minimum for $\E_\varepsilon$, we obtain the following inequality
\begin{equation}\label{MinEnergy3}
\E_\varepsilon(\rho_\infty^\varepsilon) \meni \E_\varepsilon(\rho_\infty) = \E(\rho_\infty) + \varepsilon\int \rho_\infty \log\rho_\infty.
\end{equation}
By the characterization of the minimum $\rho_\infty$ in \cite{vazquez1,chafai}, we know that $\rho_\infty \in \mathcal{P}_{ac}^2(\R)\cap L^\infty(\R)$, and hence the second term on the right hand side of~\eqref{MinEnergy3} is finite. Thus, we can take the limit $\varepsilon \rightarrow 0$ and obtain that $\limsup_{\varepsilon \to 0^+} \E_\varepsilon(\rho_\infty^\varepsilon) \meni  \E(\rho_\infty)$.

For the opposite inequality $\liminf_{\varepsilon\to 0^+}
\E_\varepsilon(\rho_\infty^\varepsilon) \maig \E(\rho_\infty)  $, we can use the fact that $\rho_\infty$ is the minimum for $\E$ and write
\begin{align}
\E(\rho_\infty) &\meni \E(\rho_\infty^\varepsilon) = \E_\varepsilon(\rho_\infty^\varepsilon) - \varepsilon \mathcal{H}(\rho_\infty^\varepsilon) + \varepsilon \pi \int x^2\rho_\infty^\varepsilon \label{MinEnergy1}\\
&\meni \E_\varepsilon(\rho_\infty^\varepsilon) + \varepsilon \pi \int x^2\rho_\infty^\varepsilon. \label{MinEnergy2}
\end{align}

So, it is sufficient to prove that the second moments of $\rho_\infty^\varepsilon$ are uniformly bounded for $\varepsilon>0$ sufficiently small. For this, note that
\begin{align*}
0 &\meni \;\; \frac{\lambda}{4}\int x^2 \rho_\infty^\varepsilon\;\; \meni \;\;\frac{(1-\varepsilon\pi)\lambda}{2}\int x^2 \rho_\infty^\varepsilon \\
&\meni \; \frac{(1-\varepsilon\pi)\lambda}{2}\int x^2 \rho_\infty^\varepsilon + \frac{c_{1,s}}{2}\int \frac{d\rho_\infty^\varepsilon(x)d\rho_\infty^\varepsilon(y)}{|x-y|^{1-2s}} +\frac{\varepsilon\lambda}{2} \mathcal{H}(\rho_\infty^\varepsilon)\\
&= \;\; \E_\varepsilon(\rho_\infty^\varepsilon) \;\;\meni\;\; \E_\varepsilon(\rho_\infty)\;\; \meni \;\;\E(\rho_\infty) + \left|\int \rho_\infty\log\rho_\infty \right|
\end{align*}
for all $0<\varepsilon<1/2\pi$. Therefore, by~\eqref{MinEnergy1} and~\eqref{MinEnergy2}
\[ \E(\rho_\infty) \meni \; \liminf_{\varepsilon \to 0^+}\E_\varepsilon(\rho_\infty^\varepsilon) +  \lim_{\varepsilon \to 0^+}\varepsilon \pi \int x^2\rho_\infty^\varepsilon  = \;\liminf_{\varepsilon \to 0^+}\E_\varepsilon(\rho_\infty^\varepsilon).\]

Hence, as $\varepsilon$ goes to zero from above, we have that the minimum of $\E_\varepsilon(\rho)$ indeed converge to the minimum of $\E(\rho)$, i.e., $\E(\rho_\infty)=\displaystyle\lim_{\varepsilon \to 0^+}\E_\varepsilon(\rho^\varepsilon_\infty)$.

\item[iii)] Finally, let us prove that $\E(\rho(t))\meni \liminf_{\varepsilon\to 0^+}\E_\varepsilon(\rho^\varepsilon(t))$, as a consequence of the convergence of $\rho^\varepsilon(t)$
to $\rho(t)$ in $\mathcal{P}_{2,ac}(\R)$ and the lower semi-continuity of the energy $\E_\varepsilon$.
For this we can use the bound~\eqref{ExpBound} to obtain
\[\lim_{R\to \infty}\;\sup_{\varepsilon>0}\;\int_{|x|> R}\rho^\varepsilon(t,x)dx \meni \lim_{R\to \infty}\;C(t)\int_{|x|> R}e^{-a(t)|x|}dx = 0,\]
which means that $\rho^\varepsilon(t)$ is a tight family of probability measures and by Prokhorov Theorem, there exist a sequence $\varepsilon_n\rightarrow 0^+$ such that $\rho^{\varepsilon_n}(t)\rightharpoonup\rho(t)$, i.e.,
\begin{equation}\label{weakconvergence}
\int_\R \varphi(x)\rho^{\varepsilon_n}(t,x)\;dx \rightarrow \int_\R \varphi(x)\rho(t,x)\;dx\;\;,\;\;\; \forall \varphi \in C_b(\R)
\end{equation}
Moreover, due to uniform exponential bound, we also have that
\begin{equation}\label{2ndmoment}
\lim_{R\to \infty}\;\sup_{\varepsilon_n\to 0}\;\int_{|x|\maig R}x^2\rho^{\varepsilon_n}(t,x)dx \meni \lim_{R\to \infty}\;C(t)\int_{|x|> R}x^2e^{-a(t)|x|}dx = 0.
\end{equation}
Therefore, by  item (iii) of~\eqref{limsup} we have that~\eqref{weakconvergence} and~\eqref{2ndmoment} imply that $\rho^{\varepsilon_n}(t)$ converges to $\rho(t)$ in $(\mathcal{P}_2(\R),W_2)$. Now, for the following inequality
\[\E(\rho^{\varepsilon_n}(t)) = \E_{\varepsilon_n}(\rho^{\varepsilon_n}(t)) - \varepsilon_n \mathcal{H}(\rho^{\varepsilon_n}(t)) +\pi\varepsilon_n\int x^2 \rho^{\varepsilon_n}(t,x) \meni \E_{\varepsilon_n}(\rho^{\varepsilon_n}) +\pi\varepsilon_n\int x^2 \rho^{\varepsilon_n}(t,x),     \]
and by the fact that $\E$ is lower semi-continuous in $(\mathcal{P}_2(\R),W_2)$ and the second moments of $\rho^{\varepsilon_n}(t)$ are uniformly bounded w.r.t $n$, we obtain
\begin{equation*}
\E(\rho(t)) \meni \liminf_{n\to\infty} \E(\rho^{\varepsilon_n}(t)) \meni \liminf_{n\to\infty} \E_{\varepsilon_n}(\rho^{\varepsilon_n}(t)).
\end{equation*}
\end{itemize}
Putting all the limits as $\varepsilon$ goes to zero together, we can conclude
the exponential convergence of $\E(\rho(t))-\E(\rho_\infty)$, that is,
\begin{align*}
\E(\rho(t))-\E(\rho_\infty) &\meni  \liminf_{n\to\infty} \E_{\varepsilon_n}(\rho^{\varepsilon_n}(t)) - \lim_{n\to\infty} \E_{\varepsilon_n}(\rho^{\varepsilon_n}_\infty) \\
&= \liminf_{n\to\infty} \Big( \E_{\varepsilon_n}(\rho^{\varepsilon_n}(t))-\E_{\varepsilon_n}(\rho^{\varepsilon_n}_\infty)\Big) \\
&\meni e^{-2\lambda t} \liminf_{n\to\infty} \Big( \E_{\varepsilon_n}(\rho_0)-\E_{\varepsilon_n}(\rho^\varepsilon_\infty)\Big) \\
&= e^{-2\lambda t}\Big( \E(\rho_0)-\E(\rho_\infty)\Big).
\end{align*}

If the regularity assumption in~\eqref{InitialRegularity} is not true, we can proceed the above argument with the mollified initial data $\rho_{0,\delta} = \eta_\delta\ast\rho_0$, which has the same bound and mass as $\rho_0$. Since we still have the same exponential bounds for the respective solutions $\rho_\delta(t)$, we can argue as above and conclude that $\E(\rho(t))\meni\liminf_{\delta\to 0}\E(\rho_\delta(t))$ holds for all $t>0$. For $t=0$ we can use the exponential bound of the initial data and the Dominated Convergence Theorem to conclude that $\lim_{\delta\to 0}\E(\rho_{\delta,0})=\E(\rho_0)$.
\end{proof}

As a direct consequence of the Talagrand inequality in \eqref{Talagrand}, we also obtain the exponential decay in Wasserstein distance.
\begin{corollary}
Assume that $\rho_0$ satisfies $0\meni \rho_0(x) \meni Ae^{-a|x|}$ for all $x \in \R$ and some $a,A \maig 0$. Then, for each $0<s <1/2$, the solution of ~\eqref{eq:fracPMESM} with initial data $\rho_0$ satisfies
\begin{equation*}
W_2(\rho(t),\rho_\infty) \meni e^{-\lambda t} \sqrt{\frac{2}{\lambda}\Big( \E(\rho_0)-\E(\rho_\infty) \Big) }.
\end{equation*}
\end{corollary}

For the Fokker-Planck equation or the classic Porous Medium Equations, exponential convergence of the relative entropy $\mathcal{E}(\rho) - \E(\rho_\infty)$ implies convergence of $\rho$ to the
steady states $\rho_\infty$ in some classical $L^p$ norms. Here we can show that the convergence in the
relative entropy implies the convergence of the norm $\| (-\Delta)^{-\frac{s}{2}}(\rho-\rho_\infty)\|_2$.

\begin{lemma}\label{lemmaE} Let $\rho_\infty$ be the unique minimizer of $\mathcal{E}$, then for any
$\rho \in \mathcal{P}_2(\mathbb{R}^d)$,
\begin{equation*}
\frac{1}{2}\|(-\Delta)^{-\frac{s}{2}}(\rho-\rho_\infty)\|_2^2
\leq \mathcal{E}(\rho)-\mathcal{E}(\rho_\infty).
\end{equation*}
\end{lemma}

\begin{proof} The characterization~\eqref{min1} and~\eqref{min2} of the global minimizer $\rho_\infty$
and the non-negativity of $\rho-\rho_\infty$ outside of the support of $\rho_\infty$ imply that
\[
0 = C_*\int_{\mathbb{R}^d} (\rho - \rho_\infty)
\leq \int_{\mathbb{R}^d} \left((-\Delta)^{-s}\rho_\infty(x) + \lambda\frac{|x^2|}{2}\right)(\rho-\rho_\infty).
\]
Therefore, we deduce
\begin{align*}
\mathcal{E}(\rho)-\mathcal{E}(\rho_\infty)
&= \frac{1}{2}\int \rho (-\Delta)^{-s}\rho
-\frac{1}{2}\int \rho (-\Delta)^{-s}\rho_\infty
+\frac{\lambda}{2}\int |x|^2(\rho-\rho_\infty) \cr
&\geq \frac{1}{2}\int \rho (-\Delta)^{-s}\rho
-\frac{1}{2}\int \rho (-\Delta)^{-s}\rho_\infty
-\int (\rho-\rho_\infty) (-\Delta)^{-s}\rho_\infty \cr
&= \frac{1}{2}\int (\rho-\rho_\infty)(-\Delta)^{-s}(\rho-\rho_\infty)
=\frac{1}{2}\|(-\Delta)^{-\frac{s}{2}}(\rho-\rho_\infty)\|_2^2.
\end{align*}
\end{proof}

Since the norm $\|(-\Delta)^{-\frac{s}{2}}(\rho-\rho_\infty)\|_2$ is in general weak, it is unlikely to produce a bound on any stronger $L^p$ norm for the difference  $\rho-\rho_\infty$. One way to show the exponential
convergence of $\rho(t)$ to $\rho_\infty$ is by assuming that a higher
norm on $\rho-\rho_\infty$ is bounded. For example, if
$\|(-\Delta)^{\frac{s}{2}}(\rho-\rho_\infty)\|_2$ is uniformly bounded,
then we have (easy to establish in Fourier space)
\[
 \|\rho - \rho_\infty\|_2^2
 \leq \|(-\Delta)^{\frac{s}{2}}(\rho-\rho_\infty)\|_2
 \|(-\Delta)^{-\frac{s}{2}}(\rho-\rho_\infty)\|_2
\]
and $\|\rho - \rho_\infty\|_2$ converges to zero also
exponentially fast, but with a smaller rate.

Let us prove that in fact the exponential convergence also holds in $L^2$ without any additional hypothesis. For this we shall use the following interpolation inequality.

\begin{theorem}\label{Theorem2}
Let $0 < \alpha\meni 1$ and $0< s < d/2$ and $0<r<\alpha/2$. There exists a constant $C = C(d,s,\alpha)$ such that
\begin{equation}\label{interp}
\norm{u}_2 \meni C \|(-\Delta)^{-\frac{s}{2}}u\|_2^{\sigma_1} \;[u]_\alpha^{\sigma_2} \;\norm{u}_1^{\sigma_3}
\end{equation}
for all $u \in L^1(\R^d)\cap C^\alpha(\R^d)$ with
\[
\sigma_1 =\frac{r}{s+r} \,,\qquad \sigma_2 = \frac{s(d+2r)}{2(d+\alpha)(s+r)} \,,\qquad \sigma_3 = \frac{s(d+2\alpha-2r)}{2(d+\alpha)(s+r)}.
\]
\end{theorem}

\begin{proof}
We first use Fourier variables, Plancherel's formula, and the H\"older's inequality to interpolate between $\dot{H}^{r}(\R^d)$ and $(-\Delta)^{-\frac{s}{2}}u \in L^2(\R^d)$ obtaining
\begin{align}
\norm{u}_2^2 = \int_{\R^d} |\widehat u(\xi)|^2d\xi\le&\, \left(\int_{\R^d} |\widehat u(\xi)|^2|\xi|^{-2s}d\xi\right)^{\sigma_1}\left(\int_{\R^d} |\widehat u(\xi)|^2|\xi|^{2r}d\xi\right)^{1-\sigma_1} \nonumber \\
=&\, \|(-\Delta)^{-\frac{s}{2}}u\|_2^{2\sigma_1}\left(\int_{\R^d} |\widehat u(\xi)|^2|\xi|^{2r}d\xi\right)^{1-\sigma_1}\label{interfourier}
\end{align}
where $\sigma_1=r/(s+r)$, for all $0<s<1/2$ and $r>0$.

Our aim now is to bound $\dot{H}^{r}(\R^d)$ by $[u]_\alpha$ and $\|u\|_1$. We write the singular integral representation of this norm (Proposition 3.4 of~\cite{Hitchhiker}) and we split it as
\begin{align*}
\|u\|_{\dot{H}^r}^2=\int_{\R^d} |\widehat u(\xi)|^2|\xi|^{2r}d\xi= &\,C_{d,r}\int_{\R^d}\int_{\R^d} \frac{(u(x)-u(y))^2}{|x-y|^{d+2r}}\,dxdy\\
=&\,C_{d,r}\iint_{|x-y|\le R} \frac{(u(x)-u(y))^2}{|x-y|^{d+2r}}\,dx dy +
C_{d,r}\iint_{|x-y|> R} \frac{(u(x)-u(y))^2}{|x-y|^{d+2r}}\,dx dy\\
:=&\,I_1+I_2\,.
\end{align*}
To estimate $I_1$, we make use of $|u(x)-u(y)|\le [u]_\alpha \, |x-y|^{\alpha}$ to get, by the change of variables $(z,w)=(x-y,x+y)$, that
\begin{align*}
I_1&=C_{d,r}\iint_{|x-y|\le R} \frac{(u(x)-u(y))^2}{|x-y|^{d+2r}}\,dxdy\le
C_{d,r}[u]_\alpha \iint_{|x-y|\le R} \frac{|u(x)-u(y)|}{|x-y|^{d+2r-\alpha}}\,dxdy\\
&\le C [u]_\alpha \norm{u}_1\,  \int_{|z|\le R} |z|^{\alpha-2r-d}\,dz \leq C [u]_\alpha\|u\|_1 R^{\alpha-2r}\,,
\end{align*}
where the last step is allowed since $2r<\alpha$. On the other hand, we can similarly estimate the far field term as
$$
I_2=C_{d,r}\iint_{|x-y|\ge R} \frac{(u(x)-u(y))^2}{|x-y|^{d+2r}}\,dxdy\le 4C_{d,r}\int_{\R^d} |u(x)|^2dx\,\int_{|z|\ge R}\frac{dz}{|z|^{d+2r}} \le C \|u\|^2_2 R^{-2r}\,.
$$
Joining the two integrals and optimizing in $R$, we infer
\begin{equation}\label{eqq.3}
\|u\|_{\dot{H}^r}^2 \le C \|u\|_2^{2(\alpha-2r)/\alpha}\|u\|_1^{2r/\alpha}[u]_\alpha^{2r/\alpha}.
\end{equation}
We finally use the classical interpolation results between $L^p(\R^d)$ and $C^\alpha(\R^d)$ spaces due to L.~Nirenberg in \cite{Nirenberg}, see also \cite{BBDGV} for a full statement. This interpolation inequality ensures the existence of a constant depending on $\alpha$ and $d$ such that
$$
\norm{u}_2^2 \le C\norm{u}_1^{(d+2\alpha)/(\alpha+d)} [u]_\alpha^{d/(\alpha+d)}\,.
$$
Putting it together with \eqref{eqq.3}, it yields
$$
\|u\|_{\dot{H}^r}^2 \le C \|u\|_1^{(d+2\alpha-2r)/(d+\alpha)}[u]_\alpha^{(d+2r)/(d+\alpha)}\,.
$$
Finally, we plug this into \eqref{interfourier} to conclude \eqref{interp}.
\end{proof}

Therefore, from Theorem~\ref{Theorem} and Theorem~\ref{Theorem2}, we derive the following decay towards the stationary state under the $L^2$ norm.

\begin{corollary}\label{decayl2}
Assume that $\rho_0$ satisfies $0\meni \rho_0(x) \meni Ae^{-a|x|}$ for all
$x \in \R$ and some $a,A \maig 0$. Then, for each $0<s <1/2$, the solution of ~\eqref{eq:fracPMESM} with initial data $\rho_0$ satisfies
\[
\norm{\rho(t)-\rho_\infty}_2 \meni C \left(1 + [\rho_\infty]_\alpha \right)^{\sigma_2}\left(\E(\rho_0)-\E(\rho_\infty) \right)^{\frac{\sigma_1}{2}} e^{-\lambda\sigma_1 t} \,.
\]
\end{corollary}

\begin{proof}
Given $\rho_0$ under the conditions above, we know from Theorem 5.1 of~\cite{vazquez3} that there exists an $\alpha \in (0,1)$ such that the solution $\rho$ of~\eqref{eq:fracPMESM} satisfies $\rho(t) \in C^\alpha(\R)$ for all $t>0$ with a uniform bound in time. Since $\rho_\infty$ is $(1-s)$-H\"{o}lder continuous, we can use inequality \eqref{interp} for $u=\rho(t)-\rho_\infty$ and $0<r<2\min(\alpha,1-s)$ to conclude.
\end{proof}

Let us point out that the decay of the entropy in Theorem \ref{Theorem} implies a uniform in time control of the second moment of the solutions trivially at least for $0<s<1/2$. Otherwise, one has to work a bit due to the sign of the constant in the fractional operator. In any case, a uniform in time control of the second moments together with the $L^2$-decay rates implies $L^1$-decay rates of the form
\begin{align}\label{l1decay}
\norm{\rho(t)-\rho_\infty}_1 &\meni \int_{|x|<R} |\rho(t,x)-\rho_\infty(x)|dx +
\int_{|x|\geq R} |\rho(t,x)-\rho_\infty(x)|dx \cr
& \meni C\left( R^{d/2}\norm{\rho(t)-\rho_\infty}_2 + R^{-2}\int_{\mathbb{R}^d} |x|^2\big(\rho(t,x)+\rho_\infty(x)\big) dx \right)\cr
&\meni C\big({\cal E}(\rho_0)+{\cal E}(\rho_\infty)\big)^{d/(d+4)}\norm{\rho(t)-\rho_\infty}_2^{4/(d+4)},
\end{align}
by choosing $R\sim \big(({\cal E}(\rho_0)+{\cal E}(\rho_\infty))/\norm{\rho(t)-\rho_\infty}_2\big)^{2/(d+4)}$; see a similar calculation in \cite[Lemma 2.24]{MR2355628} for instance. In one dimension, using Corollary~\ref{decayl2},
we obtain the decay rate $e^{-4\lambda\sigma_1t/5}$ for $\norm{\rho(t)-\rho_\infty}_1$.

We finally remark that the decay in $L^p$-norms obtained via Corollary \ref{decayl2} and \eqref{l1decay} are translated through the change of variables \eqref{changevariables}-\eqref{changevariables2} into algebraic decay rates toward self-similar solutions of the original fractional porous medium equation \eqref{eq:fracPME}.

%%%%%%%%%%%%%%%%%%%%%%%%%%%%%%%%%%%%%%%%%%%%%%%%%%%%%%%%%%%%%%%%%%%%%%%%%
\appendix
\section{Distribution Derivatives of Riesz potential}
\label{sec:distdev}

In the discussion of the entropy dissipation methods in Section~\ref{BE},
the explicit expression of the Hessian matrix $D_{ij}(-\Delta)^{-s}\rho$ is needed to simplify the terms in $\mathcal{R}(\rho)$. Since the Riesz potential $(-\Delta)^{-s}$ is a singular integral, these second order derivatives can not be applied to the kernel~\eqref{eq:kernelW} directly, but can be derived
from several equivalent approaches. Below, we interpret $D_{ij}(-\Delta)^{-s}\rho$ as distributional derivatives, and obtain the expressions using the definition in a similar way as representing the velocity gradient using vorticity in fluid mechanics~\cite{MR1867882}.

For any test function $\phi \in C^\infty(\mathbb{R}^d)$, the distributional derivative $D_{ij}(-\Delta)^{-s}\rho$ is defined as
\[
    \big\langle D_{ij}(-\Delta)^{-s}\rho, \phi\big\rangle :=
    \big\langle (-\Delta)^{-s}\rho,D_{ij}\phi\big\rangle
    =c_{d,s}\int_{\mathbb{R}^d}\int_{\mathbb{R}^d}
    \frac{\rho(y)}{|x-y|^{d-2s}}\frac{\partial^2\phi(x)}{\partial x_i\partial
    x_j}dydx.
\]
Next, we use integration by parts to shift the derivatives from
the test function $\phi$ to the singular integral $(-\Delta)^{-s}\rho$,
by writing the above expression as a limit outside a ball. More precisely,
\begin{align*}
\big\langle (-\Delta)^{-s}\rho,D_{ij}\phi\big\rangle &=
\lim_{\epsilon \to 0^+} c_{d,s}
\int_{\mathbb{R}^d} \rho(y) \left[\int_{B(y,\epsilon)^c} \frac{1}{|x-y|^{d-2s}}
\frac{\partial^2\phi(x)}{\partial x_i\partial x_j} dx \right]dy \cr
&= \lim_{\epsilon \to 0^+} (d-2s)c_{d,s}
\int_{\mathbb{R}^d} \rho(y) \left[\int_{B(y,\epsilon)^c} \frac{x_i-y_i}{|x-y|^{d+2-2s}}
\frac{\partial\phi(x)}{\partial x_j} dx \right]dy,
\end{align*}
where $B(y,\epsilon)^c$ is the complement of the ball $B(y,\epsilon)
=\{x\in\mathbb{R}^d\mid |x-y|<\epsilon\}$ and
 the integration on the boundary $\partial B(y,\epsilon)$ vanishes in
the limit. Integrating by parts again, we obtain (the unit outer normal
at $x\in B(y,\epsilon)^c$ is $-(x-y)/|x-y|$)
\begin{multline}\label{eq:app1st}
\lim_{\epsilon\to 0^+} c_{d,s}^+
\int_{\mathbb{R}^d} \rho(y) \left[\int_{B(y,\epsilon)^c}
K_{ij}(x-y)\phi(x)dx \right.
\left.-\int_{\partial B(y,\epsilon)}
\frac{(x_i-y_i)(x_j-y_j)}{|x-y|^{d+3-2s}}\phi(x)dS_x
\right] dy,
\end{multline}
where $c_{d,s}^+=(d-2s)c_{d,s}$ and
\[
    K_{ij}(x) = \frac{1}{d-2s}\frac{\partial^2}{\partial x_i\partial x_j} |x|^{2s-d}=\frac{1}{|x-y|^{d+2-2s}}\left(
        (d+2-2s)\frac{x_ix_j}{|x|^{2}}-\delta_{ij}\right).
\]

Since for any $x\in \partial B(y,\epsilon)$, $\phi(x)
=\phi(y) + (x-y)\cdot\nabla \phi(y) + O(|x-y|^2)$, we can
replace $\phi(x)$ by $\phi(y)$ in the boundary integral in~\eqref{eq:app1st}, i.e.,
\[
    \lim_{\epsilon\to 0^+} \int_{\partial B(y,\epsilon)}
    \frac{(x_i-y_i)(x_j-y_j)}{|x-y|^{d+3-2s}}\;\phi(x)dS_x
= \phi(y)  \lim_{\epsilon\to 0^+} \int_{\partial B(y,\epsilon)}
    \frac{(x_i-y_i)(x_j-y_j)}{|x-y|^{d+3-2s}}\;dS_x.
\]
It is easy to see that for $j\neq i$,
\[
\int_{\partial B(y,\epsilon)} \frac{(x_i-y_i)(x_j-y_j)}{|x-y|^{d+3-2s}}\;dS_x = \int_{B(y,\epsilon)^c} K_{ij}(x-y) dx
 =0,
\]
and for $j=i$,
\[
 \int_{\partial B(y,\epsilon)} \frac{(x_i-y_i)(x_i-y_i)}{|x-y|^{d+3-2s}}\;dS_x = \int_{B(y,\epsilon)^c} K_{ii}(x-y) dx
=\frac{|\mathbb{S}^{d-1}|}{d}\epsilon^{2s-2},
\]
where $|\mathbb{S}^{d-1}|$ is the area of the unit sphere $\mathbb{S}^{d-1}=\{x\in\mathbb{R}^d\mid |x|=1\}$.

Therefore, the distributional derivative $\big\langle D_{ij}(-\Delta)^{-s}\rho, \phi\big\rangle$
written as the limit~\eqref{eq:app1st} can be simplified as
\begin{align*}
\big\langle D_{ij}(-\Delta)^{-s}\rho, \phi\big\rangle& = \lim_{\epsilon \to 0^+} c_{d,s}^+
\int_{\mathbb{R}^d} \rho(y) \left[\int_{B(y,\epsilon)^c} K_{ij}(x-y)
\phi(x)dy - \phi(y)\int_{B(y,\epsilon)^c} K_{ij}(x-y)dy\right]dy\cr
&=\lim_{\epsilon \to 0^+} c_{d,s}^+\iint_{|x-y|>\epsilon}
K_{ij}(x-y)\big( \rho(y)\phi(x)-\rho(y)\phi(y)\big) dydx\cr
&= -\lim_{\epsilon \to 0^+} c_{d,s}^+\int_{\mathbb{R}^d}
\phi(x)\left[\int_{B(x,\epsilon)}K_{ij}(x-y)\big(\phi(x)-\phi(y))dy\right]dx.
\end{align*}
This implies the following singular integral represent of the Hessian
matrix of $(-\Delta)^{-s}\rho$:
\[
 D_{ij}(-\Delta)^{-s}\rho(x) = -
c_{d,s}^+\int_{\mathbb{R}^d}K_{ij}(x-y)\big(\rho(x)-\rho(y)\big)dy.
\]

In particular, we can write the fractional Laplacian $(-\Delta)^{1-s}
\rho$ as
\begin{align*}
 (-\Delta)^{1-s}\rho(x) &= -
\sum_{i=1}^d D_{ii}(-\Delta)^{-s}\rho(x)
=c_{d,s}^+\int_{\mathbb{R}^d}K_{ij}(x-y)\big(\rho(x)-\rho(y)\big)dy \\
&=c_{d,s}^+\int_{\mathbb{R}^d} \frac{\rho(x)-\rho(y)}{|x-y|^{d+2-2s}}\; dy,
\end{align*}
recovering its standard singular integral representation~\cite{MR0350027,MR0290095}.

\section*{Acknowledgement}
This work was partially done while visiting the Isaac Newton Institute, Cambridge, UK, and the authors are grateful for its hospitality. JAC acknowledges support from projects MTM2011-27739-C04-02, 2009-SGR-345 from Ag\`encia de Gesti\'o d'Ajuts Universitaris i de Recerca-Generalitat  de Catalunya, and the Royal Society through a Wolfson Research Merit Award. JAC and YH acknowledge support from the Engineering and Physical Sciences Research Council (UK) grant number EP/K008404/1. MCS acknowledges support from Coordenação de Aperfeiçoamento de Pessoal de Nível Superior - CAPES (BR) project number BEX 0024/13-9. JLV is partially supported by Spanish Project MTM2011-24696.

%============================================================================================================
%===========================================BIBLIOGRAPHY=====================================================

%\bibliography{bioFI}
%\bibliographystyle{plain}

\end{document}